\documentclass[11pt]{article}

\usepackage[utf8]{inputenc}
\usepackage[T1]{fontenc}
\usepackage{lmodern}
\usepackage{amsmath,amssymb,amsthm,mathtools}
\usepackage{mathrsfs}
\usepackage{enumitem}
\usepackage{geometry}
\usepackage[pagebackref,hidelinks]{hyperref}
\usepackage{microtype}
\usepackage{tikz}
\usepackage{graphicx}
\usetikzlibrary{arrows.meta,calc}
\newtheorem{maintheorem}{Theorem}

\newtheorem{lemma}{Lemma}[section]
\newtheorem{corollary}[lemma]{Corollary}
\newtheorem{proposition}[lemma]{Proposition}

\theoremstyle{definition}
\newtheorem{definition}[lemma]{Definition}
\newtheorem{example}[lemma]{Example}
\theoremstyle{remark}
\newtheorem{remark}[lemma]{Remark}

\newcommand{\C}{\mathbb C}
\newcommand{\R}{\mathbb R}

\newcommand{\Q}{\mathbb Q}
\newcommand{\PP}{\mathbb P}
\newcommand{\Fcal}{\mathcal F}
\newcommand{\Gcal}{\mathcal G}
\newcommand{\Sing}{\operatorname{Sing}}
\newcommand{\Sep}{\operatorname{Sep}}
\newcommand{\Sat}{\operatorname{Sat}}

\newcommand{\Holvirt}{\operatorname{Hol}^{\mathrm{virt}}}
\newcommand{\Pol}{\operatorname{Pol}}
\newcommand{\red}{\mathrm{red}}

\title{Stability and Logarithmic Foliations  
on Complex Projective Spaces}
\author{V\'ictor Le\'on 
\thanks{ILACVN--CICN, Universidade Federal da Integração Latino-Americana,
Parque Tecnológico Itaipu, Foz do Iguaçu--PR, 85867-970, Brazil.
E-mail: \href{mailto:victor.leon@unila.edu.br}{victor.leon@unila.edu.br}}
\and Bruno Sc\'ardua
\thanks{Instituto de Matem\'atica, Universidade Federal do Rio de Janeiro,
CP 68530, Rio de Janeiro--RJ, 21945-970, Brazil.
E-mail: \href{mailto:bruno.scardua@gmail.com}{bruno.scardua@gmail.com}}
}
\date{}

\begin{document}
\maketitle

\begin{abstract}
We study Lyapunov-type stability for invariant algebraic divisors of
codimension-one holomorphic foliations on complex projective spaces. Motivated
by the notion of $L$-stability introduced in \cite{LeonScardua2018} for plane
singularities, we define a projective stability condition which is compatible with
the logarithmic setting and controls leafwise holonomy along the regular part
of the invariant divisor, while discarding any stability requirement inside
prescribed neighborhoods of its singular locus.
  Our
main result then is a global projective counterpart of the local
 classification in \cite{LeonScardua2018}: for a maximal invariant algebraic divisor,
projective $L$-stability together with non-dicritical non-nodal
generalized-curve singularities on a generic plane section forces the ambient
foliation to be globally logarithmic.  As an application  we  obtain a topological rigidity
consequence on $\mathbb P^2$ for logarithmic foliations under mild generic conditions.  The proof of the main theorem is based on propagating the consequences
of the stability hypothesis through the reduction tree by means of an
explicit Dulac transport argument, together with the classification of
$L$-stable groups of germs of one-dimensional complex diffeomorphisms.
\end{abstract}

\section{Introduction and main results}\label{sec:intro}

The notion of {\em stability} is one of the main notions in the qualitative theory of
differential equations.  In its classical Lyapunov form (\cite{Lyapunov1892}), an equilibrium of a
real differential equation is {\it stable} when every sufficiently small
neighborhood contains a smaller one whose trajectories remain in the first
neighborhood for all times for which the local flow is defined; see, for instance,\cite{Perko2001}.  What matters for us is that, this definition is not about the existence of a metric estimate, but about  the
persistence of a prescribed neighborhood under the dynamics.  This
neighborhood--saturation formulation is especially natural for foliations,
where leaves, rather than parametrized trajectories, are the intrinsic
dynamical objects.

The use of invariant neighborhoods as a topological expression of
integrability already appears in the work of Cerveau and Mattei
\cite{CerveauMattei1982}.  They proved, in particular, a local criterion for
the existence of a multivalued first integral
\(f_1^{\lambda_1}\cdots f_p^{\lambda_p}\), with positive real exponents,
whose hypotheses include a fundamental system of saturated neighborhoods of
the finite union of closed leaves adherent to the singular point (\cite[Third Part, Chapter I, Theorem 1]{CerveauMattei1982}).  This
invariant-neighborhood phenomenon is closely related to the later
Lyapunov-type formulation in \cite{LeonScardua2018}, although the two criteria
are not identical: the latter is centered on the complete separatrix set and
is designed to control the transverse dynamics of a non-dicritical germ.
Theorem~\ref{thm:A} below may therefore be viewed as a global projective
continuation of this local circle of ideas.

For singular holomorphic foliations in dimension two, it is proved in \cite{LeonScardua2018} that: {\it an $L$-stable
non-dicritical plane singularity is either holomorphically integrable or
real-logarithmic}.  Theorem~\ref{thm:A} below deals with a slightly different 
notion of stability, appropriate to the global projective case. Let us make the two notions clear, as well as their differences. 

Let $(\Fcal,0)$ be a non-dicritical germ of holomorphic foliation on
$(\C^2,0)$, let $U$ be a sufficiently small representative, and denote by
$\Sep(\Fcal,U)$ the finite analytic union of its local separatrices in $U$.

\begin{definition}[$L$-stability, \cite{LeonScardua2018}]\label{def:LS-L}
The germ $(\Fcal,0)$ is \emph{$L$-stable} if, for every sufficiently small
representative $U$ and every neighborhood $W\subset U$ of
$\Sep(\Fcal,U)$, there exists a neighborhood $V$ of $\Sep(\Fcal,U)$ such
that
\[
 \Sep(\Fcal,U)\subset V\subset W,
 \qquad
 \Sat_{\Fcal}^{W}(V)\subset W.
\]
Here $\Sat_{\Fcal}^{W}$ denotes the saturation used in
\cite[Definition~3.1]{LeonScardua2018}, namely the saturation relative to the
restriction of the foliation to the prescribed neighborhood $W$.  We shall
keep this local notion distinct from the stronger full-saturation condition
introduced below.
\end{definition}

For some of our local arguments it is useful to impose a strictly stronger
closed-set condition, in which the full saturation in a fixed representative
is required to remain in the prescribed neighborhood.

\begin{definition}[Strong closed-set $L$-stability]\label{def:relative-L}
Let $A$ be a closed $\Fcal$-invariant analytic set in a fixed representative.
We say that $A$ is \emph{strongly $L$-stable for $\Fcal$} if for every
neighborhood $W$ of $A$ there exists a neighborhood $V$ of $A$ such that
\[
 A\subset V\subset W,
 \qquad
 \Sat_{\Fcal}(V)\subset W,
\]
where the saturation is taken in the fixed ambient representative, rather
than relative to $W$.  For a germ $(\Fcal,0)$ on $(\C^2,0)$ and a union
$S\subset\Sep(\Fcal)$ of separatrices, all neighborhoods are understood to
contain the singular point.
\end{definition}

The following notion of stability  is motivated by the following remark:
Let  $\Fcal$ be a  {\it logarithmic} foliation on $\mathbb P^n$.  In homogeneous coordinates, 
$\Fcal$ is given by a closed rational 1-form 
\(
 \Omega=\sum_{j=1}^r\lambda_j\frac{dF_j}{F_j}.
\)
The $F_j$ are homogeneous polynomials and the radial saturation equation is 
\(
 \qquad
 \sum_{j=1}^r\lambda_j\deg F_j=0.
\)
In particular, we cannot have $\lambda_j >0, \forall j$. 

Let $D\subset\PP^n$ be a reduced invariant algebraic divisor and put
\(
 \Sigma_D=D\cap\Sing(\Fcal).
\)
If $N$ is a neighborhood of $\Sigma_D$, we write
$\Sat_{\Fcal}^{\,\PP^n\setminus\overline N}$ for saturation by the
restriction of the foliation to $\PP^n\setminus\overline N$.

\begin{definition}[Projective $L$-stability]\label{def:projective-L}
We say that $D$ is \emph{projectively $L$-stable} for $\Fcal$ if, for every
neighborhood $W$ of $D$ and every neighborhood $N$ of $\Sigma_D$ with
$\overline N\subset W$, there exists a neighborhood $V$ of $D$ satisfying
\[
 D\subset V\subset W,
 \qquad
 \Sat_{\Fcal}^{\,\PP^n\setminus\overline N}
       (V\setminus\overline N)\subset W.
\]
\end{definition}

Thus projective $L$-stability controls leafwise transport along the regular
part of the divisor but imposes no stability condition inside the deleted
singular neighborhoods.  In particular, projective $L$-stability is not a
stability condition for the singularities of the foliation: it is a relative
stability condition for leafwise transport along the regular part of the
invariant divisor, after arbitrarily prescribed neighborhoods of its singular
locus have been removed.  
Example~\ref{ex:coordinate-triangle} shows that projective stability is genuinely weaker than
strong closed-set stability at the crossings.

Section~\ref{sec:log-characterization} shows that
this weakening is genuine by characterizing it inside the normal-crossing
logarithmic class and by comparing it with the strong closed-set notion above.

A technical point is needed in the singular reduction used later.  A reduced
nondegenerate plane singularity written as
\[
 dy/y-\lambda \,dx/x=0
\]
is called \emph{nodal} when $\lambda\in\R_{>0}\setminus\Q$; such a
singularity has a nodal separator and does not admit the local Dulac passage
between the two axes that is used in our propagation argument. 
Let us illustrate the preceding notions and their differences in a concrete model:
For the linear germ
\[
x\,dy-\lambda y\,dx=0, \, \lambda \in \mathbb C\setminus \{0\}, 
\]
non-dicriticality means that $\lambda \notin \mathbb Q_{>0}$. In this case, 
the following characterizations hold:
\[
(\mathcal F,0)\ \text{is  $L$-stable (in the sense of Definition~\ref{def:LS-L})}
\quad\Longleftrightarrow\quad
\lambda\in\mathbb Q_{<0}\cup(\mathbb R\setminus\mathbb Q),
\]
\[
\{xy=0\}\ \text{is strongly $L$-stable (Definition~\ref{def:relative-L})}
\quad\Longleftrightarrow\quad
\lambda\in\mathbb R_{<0},
\]
whereas, viewed as a normal-crossing logarithmic model in the projective
setting,
\[
\{xy=0\}\ \text{is projectively $L$-stable (Definition~\ref{def:projective-L})}
\quad\Longleftrightarrow\quad
\lambda\in\mathbb R.
\]
Thus the positive irrational nodal case is 
$L$-stable and projectively $L$-stable, but not strongly $L$-stable.
 We shall say
that a generalized-curve singularity is \emph{non-nodal} if no nodal
singularity occurs in a reduction.  This condition is independent of the
choice of a minimal reduction,  and is 
needed in Section~\ref{sec:transverse-propagation}.

Our  first main result is a characterization of certain  logarithmic projective foliations in terms of projective stability and reads as follows. 
\begin{maintheorem}[Projective logarithmicity]\label{thm:A}
Let $\Fcal$ be a codimension-one holomorphic foliation on
$\PP^n$, $n\ge2$, and let
\(
 D=D_1\cup\cdots\cup D_r
\)
be a nonempty reduced $\Fcal$-invariant algebraic divisor.  Assume that:
\begin{enumerate}[label=\textup{(\roman*)}]
\item\label{PAstable} $D$ is projectively $L$-stable;
\item\label{PAgen} for a sufficiently general linear plane
$E\simeq\PP^2$, every singularity of $\Fcal|_E$ lying on $D\cap E$ is a
non-dicritical non-nodal generalized-curve singularity;
\item\label{PAmax} $D$ is maximal among reduced invariant algebraic divisors
of $\Fcal$, equivalently every invariant algebraic hypersurface of $\Fcal$
is contained in $D$.
\end{enumerate}
Then $\Fcal$ is globally logarithmic.  More precisely, if
$D_j=\{F_j=0\}$ with $F_j$ reduced and homogeneous, there exist nonzero
constants $\lambda_j\in\C$ such that $\Fcal$ is defined on $\PP^n$ by
\[
 \Omega=\sum_{j=1}^r\lambda_j\frac{dF_j}{F_j},
 \qquad
 \sum_{j=1}^r\lambda_j\deg F_j=0.
\]
In particular $\Pol(\Omega)_{\red}=D$.
\end{maintheorem}

Theorem~\ref{thm:A} is naturally viewed as a global projective
counterpart of \cite{LeonScardua2018} stated for germs.   
Theorem~\ref{thm:A}  has a natural topological equivalence application in dimension
two.
The question of topological rigidity of generic projective foliations in $\mathbb P^2$ has originally been addressed by Ilyashenko (\cite{Ilyashenko1978,Ilyashenko2007}). 
 This question was subsequently revisited in
\cite{TeymuriGarakaniMafraScardua2013}, where logarithmic foliations
arise naturally among the exceptional non-rigid projective models.
The question we address is slightly different:
{\em Given a projective foliation topologically equivalent to a  logarithmic foliation, is it also logarithmic, i.e., is the class of logarithmic foliations, topologically rigid?}

In Theorem~\ref{thm:B}  we give a positive answer to this question. This is done by 
 combining  techniques from Theorem~\ref{thm:A}, topological invariance and equidesingularization
of generalized-curve singularities in 
\cite{CamachoLinsNetoSad1984} and  the  global logarithmicity theorem
in \cite{Licanic2000}. 

\begin{maintheorem}[Topological rigidity of the logarithmic class]\label{thm:B}
Let $\Gcal$ and $\Fcal$ be holomorphic foliations on $\PP^2$, and let
\[
 h:(\PP^2,\Gcal)\longrightarrow(\PP^2,\Fcal)
\]
be an ambient topological conjugacy between them.  Assume that $\Gcal$ is a
global logarithmic foliation defined by
\(
 \Omega_0=\sum_{j=1}^r\lambda_j\frac{dF_j}{F_j},
 \qquad \lambda_j\in\C^*,
 \qquad \sum_{j=1}^r\lambda_j\deg F_j=0,
\)
whose reduced polar divisor
\(
 D=\bigcup_{j=1}^r(F_j=0)
\)
has normal crossings.  Assume moreover that every crossing is non-dicritical;
equivalently,
\[
 \frac{\lambda_i}{\lambda_j}\notin\Q_{<0}
 \qquad\text{whenever }D_i\cap D_j\neq\varnothing.
\]
Then $\Fcal$ is a  logarithmic foliation on $\PP^2$.  Moreover, 
$D'=h(D)$ is a reduced invariant algebraic normal-crossings divisor, the local
branches of $D'$ form the complete separatrix set of $\Fcal$ at every
singular point of $\Fcal$ on $D'$ and  those singularities are generalized curves.
\end{maintheorem}

 Theorem~\ref{thm:B} gives the following
purely logarithmic rigidity statement.

\begin{corollary}\label{cor:residue-rigidity}
Let $\Gcal$ be a logarithmic foliation on $\PP^2$ defined by
\[
 \Omega_0=\sum_{j=1}^r\lambda_j\frac{dF_j}{F_j},
 \qquad \lambda_j\in\C^*,
 \qquad \sum_{j=1}^r\lambda_j\deg F_j=0,
\]
and suppose that its reduced polar divisor has normal crossings.  Assume that
\[
 \frac{\lambda_i}{\lambda_j}\in\R\setminus\Q
 \qquad\text{whenever} \, i\ne j, \, D_i\cap D_j\neq\varnothing.
\]
A holomorphic foliation on $\PP^2$ which is topologically
conjugate to $\Gcal$, is also  logarithmic.
\end{corollary}

The same argument is not specific to the projective plane.  In the simplest
surface case covered by \cite{Licanic2000} one obtains the following extension.

\begin{corollary}[Picard-rank-one algebraic surfaces]\label{cor:surface-rigidity}
Let $X$ be a smooth compact algebraic surface with
$\operatorname{Pic}(X)\simeq\mathbb Z$.  Let $\Gcal$ and $\Fcal$ be
holomorphic foliations on $X$ which are ambiently topologically conjugate.
Assume that $\Gcal$ is logarithmic and admits a compact reduced
normal-crossings invariant polar divisor $D$ such that every singularity of
$\Gcal$ on $D$ is a non-dicritical generalized-curve singularity and the
local branches of $D$ form its complete separatrix set.  Then $\Fcal$ is
logarithmic on $X$.
\end{corollary}

The remainder of the paper is organized as follows.  Section~2 records an
auxiliary observation for the stronger closed-set notion.  Section~\ref{sec:log-characterization}
analyzes projective $L$-stability in the logarithmic class and compares it
with strong closed-set stability.  Section~\ref{sec:transverse-propagation}
contains the Dulac transport argument that replaces any
separatrix-completeness hypothesis in Theorem~\ref{thm:A}.  We then prove
Theorem~\ref{thm:A}.  The final section derives the topological rigidity
 and proves the residue
and surface corollaries.

\section{A local propagation tool}

Before turning to the projective arguments, we isolate a simple observation
about strong closed-set stability at a non-dicritical plane singularity.  The
 purpose is to identify
when a strongly stable collection of separatrices is necessarily complete.
It also clarifies the distinction between full-saturation stability and the
projective notion used in Theorem~\ref{thm:A}.

\begin{proposition}[Strong stability forces the complete separatrix set]\label{prop:complete-separatrices}
Let $(\Fcal,0)$ be a non-dicritical holomorphic foliation germ on
$(\C^2,0)$ and let
\[
 \varnothing\neq S\subset\Sep(\Fcal)
\]
be a union of separatrices.  If $S$ is strongly $L$-stable for $\Fcal$ in the sense
of Definition~\ref{def:relative-L}, then
\[
 S=\Sep(\Fcal).
\]
Consequently $(\Fcal,0)$ is $L$-stable in the usual sense of
\cite{LeonScardua2018}.
\end{proposition}
\begin{lemma}[completeness: a missing separatrix prevents strong closed-set stability]\label{lem:missing-sep}
Let $(\Fcal,0)$ be a holomorphic foliation germ on $(\C^2,0)$ with an irreducible separatrix $\Gamma$.  Let $A$ be a closed invariant analytic set containing $0$ but not containing $\Gamma$.  Then $A$ is not strongly $L$-stable for $\Fcal$
in the sense of Definition~\ref{def:relative-L}.
\end{lemma}

\begin{proof}
Choose a sufficiently small representative $U$ and a regular point
$q\in\Gamma\setminus(A\cup\{0\})$.  Since $A$ is closed, there is a neighborhood $W$ of $A$ in $U$ with $q\notin W$.  Shrinking a small ball about the origin if necessary, we may suppose that this ball is contained in $W$.

Let $V$ be any neighborhood of $A$ with $A\subset V\subset W$.  Because $0\in A$, the set $V$ contains a small ball about $0$.  Hence it contains a regular point $q'\in\Gamma\setminus\{0\}$ sufficiently close to the origin.  The punctured irreducible curve $\Gamma\setminus\{0\}$ is contained in one leaf of $\Fcal$ after the representative is chosen sufficiently small.  Thus the leaf through $q'$ also contains $q$.  Therefore
\[
 q\in\Sat_{\Fcal}(V)\setminus W.
\]
No neighborhood $V$ can satisfy the stability condition for this $W$.
\end{proof}

\begin{proof}[Proof of Proposition~\ref{prop:complete-separatrices}]
Suppose that $S\neq\Sep(\Fcal)$.  Since the germ is non-dicritical, the complete separatrix set is a finite union of irreducible analytic curves.  Choose an irreducible separatrix
\[
 \Gamma\subset\Sep(\Fcal),\qquad \Gamma\not\subset S.
\]
Apply Lemma~\ref{lem:missing-sep} with $A=S$.  This contradicts the assumed strong $L$-stability of $S$.  Hence
\[
 S=\Sep(\Fcal).
\]
The strong full-saturation condition is stronger than the local condition of \cite{LeonScardua2018}; hence $(\Fcal,0)$ is $L$-stable in the  sense of \cite{LeonScardua2018}.
\end{proof}

\begin{remark}
Neither the generalized-curve nor the non-nodal hypothesis is needed for Proposition~\ref{prop:complete-separatrices}.
In Theorem~\ref{thm:A} these assumptions enter only after passing to a generic plane and
reducing the singularities on the invariant curve.  In particular, Theorem~\ref{thm:A}
does not assume that the branches of $D$ form the complete local separatrix
set.  The propagation--adaptation lemma in Section~\ref{sec:transverse-propagation}
is designed precisely to avoid that additional hypothesis.
\end{remark}

\section{Projective stability in the logarithmic class}\label{sec:log-characterization}

We now compare Definition~\ref{def:projective-L} with the logarithmic models
that motivate it.  The point of the next proposition is that, once the
singular crossings have been removed, projective $L$-stability detects
exactly the neutral character of the transverse monodromy.  The proof is
included in full because this criterion will also provide a useful test
example for the distinction between projective and closed-set stability.

\begin{proposition}[Projective stability in the logarithmic class]\label{prop:log-projective-stable}
Let $\Fcal_\Omega$ be a logarithmic foliation on $\PP^n$, $n\ge2$, defined by
\[
 \Omega=\sum_{j=1}^r\lambda_j\frac{dF_j}{F_j},
 \qquad \lambda_j\in\C^*,
 \qquad \sum_{j=1}^r\lambda_j\deg F_j=0,
\]
and assume that its reduced polar divisor
\[
 D=\bigcup_{j=1}^rD_j,\qquad D_j=(F_j=0),
\]
is a normal-crossings divisor.  Then $D$ is projectively $L$-stable if and
only if
\[
 \frac{\lambda_i}{\lambda_j}\in\R
 \qquad\text{for every }i,j.
\]
Equivalently, after multiplication of $\Omega$ by a nonzero complex
constant, all residues may be chosen real.
\end{proposition}

\begin{proof}
Assume first that $D$ is projectively $L$-stable.  Fix two distinct
components $D_i,D_j$.  Since they are hypersurfaces in projective space,
$D_i\cap D_j\neq\varnothing$.  Because $D$ has normal crossings, a generic
point $p$ of an irreducible component of $D_i\cap D_j$ belongs to no third
component, and there are local coordinates $(x,y,z_3,\ldots,z_n)$ centered at
$p$ in which
\[
 D_i=(x=0),\qquad D_j=(y=0),
\]
and
\[
 \Omega=\lambda_i\frac{dx}{x}+\lambda_j\frac{dy}{y}+\eta,
\]
where $\eta$ is holomorphic and closed.  On a small transversal to $D_i$,
the holonomy obtained by a meridian in $D_i$ around $D_i\cap D_j$ has
linear multiplier
\[
 \mu_{ij}=\exp\!\left(-2\pi i\frac{\lambda_j}{\lambda_i}\right).
\]
Choose the singular neighborhood $N$ so small that a meridian representing
this holonomy is contained in $\PP^n\setminus\overline N$.  If
$\lambda_j/\lambda_i\notin\R$, then $|\mu_{ij}|\ne1$.  One of the two
sequences $\mu_{ij}^m$ or $\mu_{ij}^{-m}$ expands transverse distances.
Consequently, for every sufficiently small transverse disc, an iterate of
the meridian holonomy leaves any prescribed larger transverse disc.  This
contradicts Definition~\ref{def:projective-L}.  Hence
$\lambda_j/\lambda_i\in\R$.  Since $i,j$ were arbitrary, all residue ratios
are real.

Conversely, multiply $\Omega$ by a constant and assume that all $\lambda_j$
are real.  Let $W$ be a neighborhood of $D$ and let $N$ be a neighborhood
of $\Sigma_D$ with $\overline N\subset W$.  Choose an intermediate
neighborhood $N_0$ of $\Sigma_D$ such that
\[
 \overline{N_0}\subset N.
\]
Then
\[
 K_0=D\setminus N_0
\]
is compact and meets only the smooth regular parts of the components.  Near
a point of $K_0\cap D_i$ choose a simply connected flow box in which
$D_i=(x=0)$ and no other component of $D$ is present.  On that box one may
write
\[
 \Omega=\lambda_i\frac{dx}{x}+dH
\]
for a holomorphic function $H$.  Therefore the positive transverse function
\[
 r=|x|\exp\!\left(\frac{\operatorname{Re}H}{\lambda_i}\right)
\]
is constant along every local leaf.  On overlaps, two such functions differ
by a positive constant.  Around a loop in $D_i\setminus\Sigma_D$, the
multiplicative monodromy of the transverse coordinate is a product of
factors
\[
 \exp\!\left(-2\pi i\frac{\lambda_j}{\lambda_i}\right),
\]
possibly repeated with integer multiplicities; all have modulus one.  Thus
the local functions $r$ define a flat Hermitian transverse metric in a
tubular neighborhood of $K_0$.

Choose finitely many such flow boxes covering $K_0$, and shrink them so that
their union $T$ has closure in $W$.  Since the longitudinal ends of $K_0$
lie in $N$ (indeed $\partial N_0\subset N$), the tube may be chosen so that,
in $M_N:=\PP^n\setminus\overline N$, its remaining side boundary is a level
set of the flat transverse radius $r$.  Compactness of $K_0$ and invariance of
the transverse metric then give $\varepsilon>0$ such that
\[
 T_\varepsilon=\{r<\varepsilon\}\cap T
\]
contains $D\setminus N$ and has the following property: every leaf of
$\Fcal|_{M_N}$ meeting $T_\varepsilon$ remains in $T\subset W$ for as long
as it stays in $M_N$.  Indeed, a leaf cannot cross the side boundary, since
$r$ is constant along leaves, while any exit through a longitudinal end would
enter $N$, which has been removed from $M_N$.

Finally choose a neighborhood $V_N$ of $D\cap\overline N$ with
$V_N\subset W$ and so small that
\[
 (V_N\setminus\overline N)\subset T_\varepsilon,
\]
and put $V=T_\varepsilon\cup V_N$.  Then $D\subset V\subset W$ and
\[
 \Sat_{\Fcal}^{\,\PP^n\setminus\overline N}
       (V\setminus\overline N)\subset W.
\]
Hence $D$ is projectively $L$-stable.
\end{proof}

The situation changes completely if one asks for the stronger full-saturation
condition of Definition~\ref{def:relative-L}.  In a local normal-crossing
logarithmic model the precise criterion is positivity, not merely reality, of
the residue ratios.

\begin{proposition}[Strong closed-set stability for a logarithmic crossing]\label{prop:strong-local-log}
Let $\Fcal$ be a codimension-one foliation germ at $0\in\C^n$ defined by
\[
 \Omega=\sum_{j=1}^r\lambda_j\frac{df_j}{f_j},
 \qquad \lambda_j\in\C^*,
\]
and assume that
\[
 D=\bigcup_{j=1}^r(f_j=0)
\]
is a reduced normal-crossing divisor.  Then $D$ is strongly $L$-stable as a
closed invariant set if and only if
\[
 \frac{\lambda_i}{\lambda_j}\in\R_{>0}
 \qquad\text{for all }i,j.
\]
Equivalently, after multiplication of $\Omega$ by a nonzero complex constant,
all residues may be chosen positive real.
\end{proposition}

\begin{proof}
Choose normal-crossing coordinates
\[
 D=(x_1\cdots x_r=0).
\]
Writing $f_j=u_jx_j$, with $u_j$ a holomorphic unit, gives
\[
 \Omega=\sum_{j=1}^r\lambda_j\frac{dx_j}{x_j}+dh
\]
for a holomorphic function $h$.  Replacing, for instance,
$x_1$ by $x_1\exp(h/\lambda_1)$ preserves the divisor and reduces the form to
\[
 \Omega=\sum_{j=1}^r\lambda_j\frac{dx_j}{x_j}.
\]

Assume first that, after multiplication by a common nonzero scalar,
$\lambda_j=a_j>0$ for every $j$.  Set
\[
 \rho(x)=\prod_{j=1}^r|x_j|^{a_j},
 \qquad \rho|_D=0.
\]
On the complement of $D$,
\[
 d\log\rho=\operatorname{Re}\Omega,
\]
so $\rho$ is constant along the leaves.  Let $U$ be a sufficiently small
representative and let $W$ be a neighborhood of $D\cap U$.  The compact set
$K=\overline U\setminus W$ is disjoint from $D$, hence
\[
 m=\min_K\rho>0.
\]
For $0<\varepsilon<m$, the set
\[
 V=\{\rho<\varepsilon\}
\]
is a neighborhood of $D\cap U$, and every leaf meeting $V$ remains in
$V$.  Thus
\[
 \operatorname{Sat}_{\Fcal}(V)\subset V\subset W,
\]
which proves strong closed-set stability.

Conversely, suppose that $D$ is strongly $L$-stable.  Fix $i\ne j$ and take
a generic point of the codimension-two stratum $D_i\cap D_j$.  On a
transverse two-plane the induced foliation is
\[
 \lambda_i\frac{dx}{x}+\lambda_j\frac{dy}{y}=0.
\]
The holonomy of one axis has multiplier
\[
 \exp\!\left(-2\pi i\frac{\lambda_i}{\lambda_j}\right).
\]
Strong closed-set stability forces this transverse linear holonomy to be
Lyapunov stable in both directions, hence its multiplier has modulus one.
Therefore
\[
 \frac{\lambda_i}{\lambda_j}\in\R.
\]
Since $i,j$ were arbitrary, after multiplication by one scalar all residues
are real and nonzero.

It remains to exclude mixed signs.  If the real residues do not all have the
same sign, there exist numbers $v_j>0$ such that
\[
 \sum_{j=1}^r\lambda_jv_j=0.
\]
Choose constants $c_j\ne0$ and consider
\[
 \gamma(t)=
 \bigl(c_1e^{-v_1t},\ldots,c_re^{-v_rt},0,\ldots,0\bigr),
 \qquad t\ge0.
\]
For finite $t$ the first $r$ coordinates are nonzero, while
$\gamma(t)\to0\in D$ as $t\to+\infty$, and
\[
 \Omega(\dot\gamma(t))=-\sum_{j=1}^r\lambda_jv_j=0.
\]
Thus $\gamma$ lies in one leaf.  Choose a point $q=\gamma(0)$ and a
neighborhood $W$ of $D$ not containing $q$.  Every neighborhood $V$ of $D$
contains $\gamma(t)$ for all sufficiently large $t$, so the full saturation
of $V$ contains $q$, contradicting strong closed-set stability.  Hence all
residues have the same sign, proving the necessity.
\end{proof}

The preceding two propositions make the distinction transparent:
projective $L$-stability requires the residues to lie on one real line,
whereas strong closed-set stability requires them to lie on one positive real
ray.  In particular, the projective residue relation explains why the strong
condition is too restrictive for a connected global polar divisor on
projective space.  The following elementary example makes the difference
visible without invoking any additional classification theorem.

\begin{example}[A projectively stable divisor which is not strongly stable]\label{ex:coordinate-triangle}
On $\PP^2$ with homogeneous coordinates $[Z_0:Z_1:Z_2]$, consider
\[
 \Omega=\frac{dZ_0}{Z_0}-2\frac{dZ_1}{Z_1}+\frac{dZ_2}{Z_2}.
\]
The projective residue relation is satisfied and the polar divisor is the
coordinate triangle
\[
 D=(Z_0Z_1Z_2=0).
\]
All residue ratios are real, so Proposition~\ref{prop:log-projective-stable}
shows that $D$ is projectively $L$-stable.  At the vertex
$Z_0=Z_1=0$, however, the two relevant residues have opposite signs.
Proposition~\ref{prop:strong-local-log} therefore shows that the local divisor
is not strongly $L$-stable as a closed invariant set.  Equivalently, the
local model is analytically conjugate to
\[
 x\,dy-\frac12 y\,dx=0,
\]
which is dicritical.  Thus projective stability is genuinely weaker than
strong closed-set stability at the crossings.
\end{example}

\section{Dulac transport through a reduction}\label{sec:transverse-propagation}

Theorem~\ref{thm:A} gives no stability information inside the deleted
neighborhoods of the singular points of $D$.  To recover the local data
needed for logarithmic extension, we pass to a reduction and use the explicit
Dulac correspondence at each reduced saddle.  The argument below is one of the
technical cores of the paper: it propagates controlled transverse
neighborhoods from the strict transform of the algebraic divisor to the
adjacent exceptional components without assuming that the original divisor
contains all local separatrices.

We first observe precisely what projective stability gives on an original
component.  The natural object is the transverse leaf-equivalence relation,
rather than a choice of generators of ordinary holonomy.

Let $U$ be an open set on which a transversal $\Sigma$ meets a regular leaf
at $p$.  For $z,w\in\Sigma$ sufficiently close to $p$, write
\[
 z\sim_U w
\]
when $z$ and $w$ belong to the same leaf of $\Fcal|_U$.  We say that this
transverse leaf relation is \emph{$L$-stable} if for every neighborhood
$\Delta$ of $p$ in $\Sigma$ there is a smaller neighborhood
$\Delta'\subset\Delta$ such that
\[
 z\in\Delta',\quad z\sim_U w
 \quad\Longrightarrow\quad w\in\Delta.
\]
Any virtual holonomy group defined by this leaf relation is then $L$-stable
in the sense of \cite{LeonScardua2018}, because every pseudo-orbit remains in a
single equivalence class.

\begin{lemma}[Projective stability controls the transverse leaf relation]\label{lem:projective-holonomy}
Let $C\subset\PP^2$ be a reduced invariant algebraic curve which is
projectively $L$-stable for $\Fcal$, let $C_i$ be an irreducible component,
and let $\Sigma$ be a small transversal at a regular point
$p\in C_i^\circ=C_i\setminus\Sing(\Fcal)$.  Fix a neighborhood $N$ of
$C\cap\Sing(\Fcal)$ with $p\notin\overline N$, and put
\[
 M_N=\PP^2\setminus\overline N.
\]
Then the transverse leaf relation induced by $\Fcal|_{M_N}$ on $\Sigma$ is
$L$-stable.  Consequently the whole virtual holonomy group
\[
 \Holvirt(\Fcal|_{M_N},\Sigma,p)
\]
(and hence every one of its subgroups) is $L$-stable.
\end{lemma}

\begin{proof}
Fix a transverse disc $\Delta\Subset\Sigma$ centered at $p$.  Choose a
neighborhood $W$ of $C$ such that $\overline N\subset W$ and, after
shrinking $\Sigma$ if necessary,
\[
 W\cap\Sigma\subset\Delta.
\]
Projective $L$-stability gives a neighborhood $V$ of $C$ with
\[
 C\subset V\subset W,
 \qquad
 \Sat_{\Fcal}^{\,M_N}(V\setminus\overline N)\subset W.
\]
Let $\Delta'$ be the component of $V\cap\Sigma$ containing $p$.  If
$z\in\Delta'$ and $z\sim_{M_N}w$ with $w\in\Sigma$, then $w$ lies on the
same leaf of the restricted foliation as a point of
$V\setminus\overline N$.  Hence
\[
 w\in
 \Sat_{\Fcal}^{\,M_N}(V\setminus\overline N)\cap\Sigma
 \subset W\cap\Sigma\subset\Delta.
\]
This is precisely $L$-stability of the transverse leaf relation.  Since an
element of the virtual holonomy group preserves the same leaf relation, any
pseudo-orbit of a point of $\Delta'$ remains in its equivalence class and
therefore in $\Delta$.
\end{proof}

Before using the reduction, we also assure that stability  passes to a generic plane
section.  This is a direct consequence of the geometric definition.

\begin{lemma}[Restriction to a generic plane]\label{lem:plane-projective-stability}
Let $D\subset\PP^n$ be projectively $L$-stable for $\Fcal$, and let
$E\simeq\PP^2$ be sufficiently general so that $E$ is transverse to the
regular stratification of $D$ and $\Fcal_E=\Fcal|_E$ is well defined.  Then
$D_E=D\cap E$ is projectively $L$-stable for $\Fcal_E$.
\end{lemma}

\begin{proof}
Let $W_E$ be a neighborhood of $D_E$ in $E$ and let $N_E$ be a neighborhood
of $D_E\cap\Sing(\Fcal_E)$ with $\overline{N_E}\subset W_E$.  After
shrinking inside $E$, choose ambient neighborhoods $W$ of $D$ and $N$ of
$D\cap\Sing(\Fcal)$ such that
\[
 W\cap E\subset W_E,
 \qquad
 \overline N\cap E\subset N_E.
\]
Apply projective $L$-stability to obtain $D\subset V\subset W$ with
\[
 \Sat_{\Fcal}^{\,\PP^n\setminus\overline N}
      (V\setminus\overline N)\subset W.
\]
With $V_E=V\cap E$, every leafwise path of $\Fcal_E$ is a leafwise path of
$\Fcal$, and therefore
\[
 \Sat_{\Fcal_E}^{\,E\setminus\overline{N_E}}
      (V_E\setminus\overline{N_E})
 \subset
 \Sat_{\Fcal}^{\,\PP^n\setminus\overline N}
      (V\setminus\overline N)\cap E
 \subset W_E.
\]
Thus $D_E$ is projectively $L$-stable.
\end{proof}

We shall also use the elementary fact that $L$-stability of the transverse
leaf relation is preserved under regular holonomy transport.  Indeed, if
$h:(\Sigma,p)\to(\Sigma',p')$ is the holonomy germ associated to a fixed
path contained in a regular leaf, then $h$ is a biholomorphism of sufficiently
small transversals and carries leaf-equivalence classes to leaf-equivalence
classes.  Thus a stable neighborhood basis on $\Sigma$ is carried to one on
$\Sigma'$.

The next lemma is the local passage used at every edge of the reduction tree.
Its proof explains why the non-nodal condition appears in Theorem~\ref{thm:A}.

\begin{lemma}[Linearized Dulac passage]\label{lem:dulac-passage}
Let $(\Fcal,p)$ be a reduced nondegenerate non-dicritical and non-nodal
singularity on a complex surface with separatrices $S_1$ and $S_2$.  Assume
that the transverse leaf relation on a small transversal to $S_1$ is
$L$-stable.  Then $(\Fcal,p)$ is analytically linearizable,
and in suitable coordinates it is
\[
 x\,dy+\alpha y\,dx=0,
 \qquad \alpha>0,
 \qquad S_1=(y=0),\quad S_2=(x=0).
\]
For fixed small $a,b\ne0$, the Dulac correspondence from
$\Sigma_1=(x=a)$ to $\Sigma_2=(y=b)$ is, on every sufficiently small
simply connected sector,
\[
 \mathcal D(z)=a\left(\frac{z}{b}\right)^{1/\alpha},
\]
where $z=y|_{\Sigma_1}$.  In particular there are constants $c_1,c_2>0$
such that
\[
 c_1|z|^{1/\alpha}\le |\mathcal D(z)|\le c_2|z|^{1/\alpha},
\]
and analogous estimates hold for the inverse correspondence.  Hence a
fundamental system of arbitrarily small transverse neighborhoods on one
branch is carried to a fundamental system on the other branch.
\end{lemma}

\begin{proof}
Write the reduced germ initially as
\[
 x(1+A(x,y))\,dy-\lambda y(1+B(x,y))\,dx=0,
 \qquad A(0,0)=B(0,0)=0.
\]
The derivative of the local holonomy of $S_1$ is, up to replacing $\lambda$
by $1/\lambda$, $e^{2\pi i\lambda}$.  The local holonomy group preserves
the transverse leaf relation, hence it is $L$-stable.  It therefore cannot
be hyperbolic and it cannot contain a nontrivial element tangent to the
identity.  By the one-dimensional classification of $L$-stable groups in
\cite{LeonScardua2018}, its possibilities are finite or analytically
linearizable of circle type.  The
standard holonomy classification of reduced nondegenerate plane
singularities (see, for instance, \cite{MatteiMoussu1980}) then gives an
analytic linearization of the foliation.  Thus we may assume
\[
 x\,dy-\lambda y\,dx=0,
 \qquad \lambda\in\R\setminus\{0\}.
\]
The singularity is non-nodal, so $\lambda\notin\R_{>0}$.  Hence
$\lambda=-\alpha$ with $\alpha>0$, which gives the displayed saddle model.

The leaves of this model satisfy
\[
 yx^{\alpha}=\text{constant}.
\]
If a leaf meets $\Sigma_1$ at $(a,z)$ and $\Sigma_2$ at $(w,b)$, then
\[
 za^{\alpha}=bw^{\alpha}.
\]
Choosing a branch of the power on a simply connected sector gives
\[
 w=a\left(\frac{z}{b}\right)^{1/\alpha}.
\]
The two-sided power estimate follows immediately after the sector and the
transversals are fixed.  Solving the same relation for $z$ gives the inverse
estimate.  Therefore discs replaced, when necessary, by finitely many
sectors are transported to arbitrarily small transverse neighborhoods in
both directions.
\end{proof}

The geometry of the preceding calculation is shown in
Figure~\ref{fig:dulac-passage}.  The important feature is that a small
transverse set on one invariant branch is carried, through leaves crossing the
Dulac box, to a small transverse set on the adjacent branch.

\begin{figure}[ht]
\centering
\begin{tikzpicture}[scale=1.0,>=Latex]
  \draw[very thick,->] (-0.3,0) -- (6.2,0) node[right] {$S_1=\{y=0\}$};
  \draw[very thick,->] (0,-0.3) -- (0,4.8) node[above] {$S_2=\{x=0\}$};
  \fill (0,0) circle (2pt) node[below left] {$p$};
  \draw[dashed,thick] (4.7,0.25) -- (4.7,2.15);
  \node[right] at (4.7,1.95) {$\Sigma_1$};
  \draw[dashed,thick] (0.25,3.55) -- (2.15,3.55);
  \node[above] at (1.95,3.55) {$\Sigma_2$};
  \draw[->] (4.7,0.55) .. controls (3.6,0.65) and (2.2,1.1) .. (1.55,1.95)
                         .. controls (1.05,2.55) and (0.75,3.05) .. (0.65,3.55);
  \draw[->] (4.7,0.9) .. controls (3.55,1.0) and (2.55,1.45) .. (1.95,2.05)
                        .. controls (1.45,2.55) and (1.2,3.0) .. (1.05,3.55);
  \draw[->] (4.7,1.25) .. controls (3.75,1.35) and (2.95,1.65) .. (2.35,2.15)
                         .. controls (1.85,2.55) and (1.55,3.0) .. (1.45,3.55);
  \draw[very thick] (4.62,0.48) -- (4.78,1.32);
  \node[right] at (4.82,0.82) {$\Delta_1$};
  \draw[very thick] (0.55,3.47) -- (1.55,3.63);
  \node[above] at (1.05,3.68) {$\Delta_2$};
  \draw[->,thick] (5.25,2.45) .. controls (4.3,3.25) and (3.0,4.05) .. (1.75,4.05);
  \node at (3.55,3.75) {$\mathcal D(z)=c z^{\beta}$};
\end{tikzpicture}
\caption{Dulac passage through a reduced saddle corner.  The power-map
estimate transports a fundamental system of small transverse neighborhoods
from $S_1$ to the adjacent invariant component $S_2$, and similarly in the
reverse direction on suitable sectors.}
\label{fig:dulac-passage}
\end{figure}
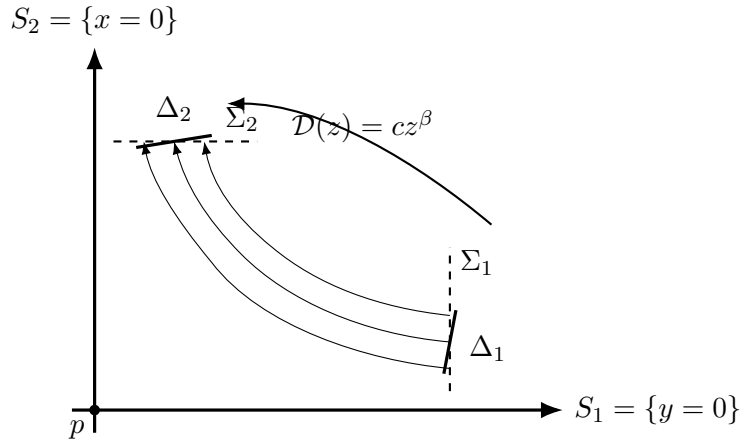

The power estimate is useful because it transfers stability of the transverse
leaf-equivalence relation itself.  No conjugation of virtual holonomy germs
by a fractional power map is needed.

\begin{lemma}[Dulac transfer of transverse leaf stability]\label{lem:dulac-transfer}
In the setting of Lemma~\ref{lem:dulac-passage}, assume that the transverse
leaf relation on a small transversal $\Sigma_1$ to $S_1$ is $L$-stable.
Then the transverse leaf relation on a small transversal $\Sigma_2$ to
$S_2$ is also $L$-stable.  Consequently the corresponding virtual holonomy
group on $\Sigma_2$ is $L$-stable.  Moreover the germ admits a closed
meromorphic defining form with simple poles along $S_1\cup S_2$.
\end{lemma}

\begin{proof}
Fix a sufficiently small Dulac neighborhood of the reduced saddle and let
$\sim_1$ and $\sim_2$ denote the leaf-equivalence relations induced on
$\Sigma_1$ and $\Sigma_2$ by the punctured foliation.  Let
$\Delta_2\Subset\Sigma_2$ be prescribed.  By the direct power estimate in
Lemma~\ref{lem:dulac-passage}, one can choose a disc
$\Delta_1\Subset\Sigma_1$ so small that every sectorial branch of the Dulac
correspondence sends $\Delta_1$ into $\Delta_2$.  Different branches differ
by local holonomy and have the same power-type control of the modulus.

By $L$-stability of $\sim_1$, choose
$\Delta_1'\Subset\Delta_1$ such that
\[
 u\in\Delta_1',\quad u\sim_1 v
 \quad\Longrightarrow\quad v\in\Delta_1.
\]
The inverse power estimate now gives a disc
$\Delta_2'\Subset\Delta_2$ such that every inverse Dulac lift of a point of
$\Delta_2'$ belongs to $\Delta_1'$.

Take $z,w\in\Sigma_2$ with $z\in\Delta_2'$ and $z\sim_2 w$.  Choose an
inverse Dulac lift $u\in\Sigma_1$ of $z$.  Following the same leaf from $w$
through the Dulac box gives an inverse lift $v\in\Sigma_1$; if a different
sectorial branch is required, $v$ differs from such a lift by local holonomy,
so in every case $u\sim_1 v$.  Hence $v\in\Delta_1$.  Applying the direct
Dulac estimate to the branch taking $v$ to $w$ gives $w\in\Delta_2$.  Thus
$\sim_2$ is $L$-stable.  Every pseudo-orbit of the virtual holonomy group
remains in one leaf-equivalence class, so that group is $L$-stable as well.

Finally Lemma~\ref{lem:dulac-passage} has already analytically linearized the
germ as
\[
 x\,dy+\alpha y\,dx=0,\qquad \alpha>0.
\]
Hence
\[
 \omega=\frac{dy}{y}+\alpha\frac{dx}{x}
\]
is a closed meromorphic defining form with simple poles precisely along
$S_1\cup S_2$.
\end{proof}

We can now iterate the preceding local passage over the finite reduction
tree.  This is the exact replacement for the separatrix-completeness
hypothesis that appears in several classical logarithmicity criteria.

\begin{lemma}[Propagation and adaptation through a non-nodal generalized-curve reduction]\label{lem:resolution-propagation}
Let $C\subset\PP^2$ be a reduced invariant algebraic curve which is
projectively $L$-stable for $\Fcal$, and assume that every singularity of
$\Fcal$ on $C$ is a non-dicritical non-nodal generalized-curve singularity.
Denote by
\[
 \pi:(\widetilde M,E)\longrightarrow(\PP^2,C)
\]
be a reduction of these singularities. Then every component of the reduced total transform carries an
$L$-stable transverse leaf-equivalence relation on sufficiently small
transversals to its regular part.  In particular its
ordinary and virtual holonomy groups are $L$-stable.  Furthermore, near every
reduced singular point of the total invariant configuration, the lifted
foliation admits a closed meromorphic defining form with simple poles along
the local invariant branches.  The same conclusion holds for the strict
transform of any additional local separatrix encountered in the reduction.
\end{lemma}

\begin{proof}
Fix a singular point $p\in C\cap\Sing(\Fcal)$.  The exceptional divisor of a
reduction over $p$ is a finite connected tree, every exceptional component is
invariant, and there are no saddle-nodes.  At least one strict transform of a
branch of $C$ meets this tree.  Choose the singular neighborhood $N$ in
Lemma~\ref{lem:projective-holonomy} to be a sufficiently small ball around
$p$ (together with disjoint small neighborhoods of the other singular points
of $C$).  Outside the exceptional divisor the blow-up is a biholomorphism, so the
$L$-stable transverse leaf relation supplied by that lemma lifts to a
transversal on the strict transform outside the exceptional divisor.  Along
the regular part of that strict transform, choose a fixed path ending at a
transversal arbitrarily close to the first reduced corner.  By regular
holonomy transport, as observed above, $L$-stability of the transverse leaf
relation is preserved along this path.  We therefore obtain the required
stable transverse relation on the strict-transform side of the first Dulac
box.

Consider the first reduced corner met when one enters the exceptional tree
from this strict transform.  The singularity is non-nodal by hypothesis.
Lemma~\ref{lem:dulac-transfer} therefore transfers the $L$-stable leaf
relation to the adjacent exceptional component and constructs an adapted
closed meromorphic form at the corner.  Once this property has been
established on that component, the same argument can be applied at each of
its reduced corners.  Proceeding inductively with the graph distance from
the original strict transform reaches every vertex of the finite tree.  No
cycle ambiguity occurs in this induction because the reduction graph over
$p$ is a tree.

If a terminal reduced singularity has a second separatrix which is not part
of the exceptional divisor and is not the strict transform of a branch of
$C$, Lemma~\ref{lem:dulac-transfer} applies once more and transfers the same
leaf-relation stability to that additional separatrix.  The adapted form
supplied at every step has simple poles along the two local invariant
branches.  Repeating the construction over the finitely many singular points
of $C$ proves the statement.
\end{proof}

\begin{lemma}[Componentwise adapted forms]\label{lem:component-adapted}
Under the hypotheses of Lemma~\ref{lem:resolution-propagation}, let $P$ be an
irreducible component of the reduced total transform.  There exists a
neighborhood $U_P$ of $P$ and a closed meromorphic one-form $\omega_P$ on
$U_P$ which defines the lifted foliation outside its polar set and has only
simple poles along the invariant branches meeting $P$.  If two components
$P$ and $Q$ meet at a reduced corner, their adapted forms satisfy
\[
 \omega_P=c_{PQ}\,\omega_Q
\]
on a connected neighborhood of the corner, for some $c_{PQ}\in\C^*$.
\end{lemma}

\begin{proof}
The ordinary holonomy group of $P^\circ=P\setminus\Sing(\widetilde\Fcal)$ is
finitely generated.  Every one of its pseudo-orbits is contained in a
single class of the transverse leaf relation.  Lemma~\ref{lem:resolution-propagation}
therefore makes the whole ordinary holonomy group $L$-stable, and the group
theorem of \cite{LeonScardua2018} shows that it is either finite cyclic or
analytically linearizable of circle type.  Choose a transversal $\Sigma$ and a coordinate
$z$ in which the whole ordinary holonomy group is linear.  In both cases the
meromorphic differential $dz/z$ is invariant under holonomy.  It therefore
suspends along $P^\circ$ to a closed meromorphic form $\omega_P^\circ$ in a
tubular neighborhood of $P^\circ$, with a simple pole along $P^\circ$.

Let $q\in P\cap\Sing(\widetilde\Fcal)$.  Lemma~\ref{lem:resolution-propagation}
provides a local closed meromorphic defining form $\omega_q$ with simple
poles along the two local invariant branches.  We normalize it so that its
restriction to a nearby transversal to $P$ is $dz/z$.  In the circle-type
case this normalization is unique because a germ commuting with a
nonresonant rotation is linear.  In the finite case, if the holonomy has
order $m$, a primitive local first integral may be chosen so that its
restriction to the transversal is $z^m$; then $(1/m)dF/F$ restricts to
$dz/z$.  Hence $\omega_q$ and $\omega_P^\circ$ agree on the overlap near a
regular point of $P$ close to $q$.  They therefore glue.  Repeating this at
the finitely many punctures extends $\omega_P^\circ$ to the asserted form
$\omega_P$ on a neighborhood of the whole component.

Finally suppose $P$ and $Q$ meet at a corner.  In the linearizing coordinates
of Lemma~\ref{lem:dulac-passage}, both adapted forms are closed logarithmic
forms defining the same linear foliation and having nonzero simple residues
on the two axes.  Their quotient is a holomorphic first integral.  The
normalizations on transversals to $P$ and $Q$ force this quotient to be
constant; equivalently, a direct comparison of the two logarithmic normal
forms gives $\omega_P=c_{PQ}\omega_Q$ with $c_{PQ}\ne0$.
\end{proof}

\begin{remark}
The proof uses the algebraic divisor only to provide the initial stable
transverse side of each reduction tree.  Additional local separatrices are
allowed and are discovered dynamically during the reduction.  
\end{remark}

\section{Proof of Theorem A}\label{sec:proof-projective}

We now assemble the preceding local and transverse ingredients.  The proof is
first carried out on the projective plane, where the adapted forms can be
globalized by the extension argument of Camacho--Lins Neto--Sad, and is then
lifted to higher dimension by restriction to a general plane and projective
extension.

We first treat the case $n=2$.  Put $C=D\subset\PP^2$ and
$\Sigma=C\cap\Sing(\Fcal)$.  Lemma~\ref{lem:projective-holonomy} converts the
geometric saturation condition~\ref{PAstable} into $L$-stability of the
transverse leaf relation, and hence of the ordinary and virtual holonomy,
along every regular component $C_i^\circ$.  Hypothesis~\ref{PAgen} says
that the singularities on $C$ are non-dicritical non-nodal generalized curves.  No
assumption is made on possible additional local separatrices.

Resolve these singularities.  Lemma~\ref{lem:resolution-propagation} shows
that the transverse stability propagates to every exceptional component and
to every additional local separatrix met by the reduced configuration.
Lemma~\ref{lem:component-adapted} then constructs an adapted closed
meromorphic form in a neighborhood of each component, with simple poles, and
shows that the forms attached to two components meeting at a corner differ
by a nonzero constant.  This is exactly the compatibility mechanism needed
in the adapted-form construction of Camacho--Lins Neto--Sad
\cite{CamachoLinsNetoSad1992}.

The projective case differs from the tree situation in an important and
useful way: no acyclicity of the incidence graph is needed.  Let $\eta$ be
a global rational one-form defining $\Fcal$ on $\PP^2$.  If $\omega_i$ is
an adapted closed form near one resolved component, write
\[
 \pi^*\eta=f_i\omega_i .
\]
On an overlap,
\[
 \omega_i=c_{ij}\omega_j,\qquad c_{ij}\in\C^*,
\]
and consequently
\[
 \frac{df_i}{f_i}=\frac{df_j}{f_j}.
\]
Thus the logarithmic differentials glue even when the divisor graph
contains cycles.  The extension argument of
\cite{CamachoLinsNetoSad1992}, followed by descent through
the resolution, yields a closed rational one-form $\Omega$ on $\PP^2$
defining $\Fcal$ and having only simple divisorial poles.  At this stage an
additional local separatrix not contained in $C$ may a priori occur as a
local branch of the polar set.  Since $\Omega$ is now rational on all of
$\PP^2$, such a branch is contained in a global irreducible polar curve.
Every irreducible component of the polar divisor of a closed rational
defining form is invariant by the foliation.  By the maximality assumption
\ref{PAmax}, no such polar component can occur outside $D$.  Conversely the
construction has a simple pole along every component of $D$.  Hence
\[
 \Pol(\Omega)_{\red}=D.
\]
A closed rational one-form on $\PP^2$ with reduced polar divisor $D$ is
logarithmic, so
\[
 \Omega=\sum_{j=1}^r\lambda_j\,\frac{dF_j}{F_j}
\]
for constants $\lambda_j\in\C^*$.  Homogeneity gives the projective residue
relation
\[
 \sum_{j=1}^r\lambda_j\deg F_j=0.
\]

Now let $n\ge3$ and choose a sufficiently general linear plane
$E\simeq\PP^2$.  By Lemma~\ref{lem:plane-projective-stability}, the curve
$D_E=D\cap E$ is projectively $L$-stable for the induced foliation
$\Fcal_E=\Fcal|_E$.  Hypothesis~\ref{PAgen} gives the required
non-dicritical non-nodal generalized-curve singularities on $D_E$.  Applying the
preceding two-dimensional construction on the plane gives a closed
meromorphic one-form defining $\Fcal_E$ near $D_E$, with simple poles.  No
separatrix-completeness condition on $D_E$ is used.

We then invoke the projective extension theorem of
Camacho--Lins Neto--Sad \cite{CamachoLinsNetoSad1992}, namely the extension
result used there for adapted meromorphic forms from a sufficiently general
plane section to the ambient projective foliation.  This is precisely the
extension step needed here.  Applied to the form obtained above on $E$, it gives a closed meromorphic defining form in
an ambient neighborhood of the section; the remaining extension across the
singular set is the standard Levi--Hartogs continuation used in that
argument.  Consequently one obtains a global closed rational one-form
$\Omega$ on $\PP^n$ defining $\Fcal$.  As in the plane case, every polar hypersurface is
invariant; maximality of $D$ therefore implies
\[
 \Pol(\Omega)_{\red}=D.
\]
Since the poles are simple, $\Omega$ is logarithmic.  Writing
$D_j=\{F_j=0\}$ gives
\[
 \Omega=\sum_{j=1}^r\lambda_j\,\frac{dF_j}{F_j},
 \qquad
 \sum_{j=1}^r\lambda_j\deg F_j=0.
\]
This proves Theorem~\ref{thm:A}.

\section{Topological rigidity and residue consequences}\label{sec:proof-rigidity}

We now prove Theorem~\ref{thm:B}.
The proof has two independent ingredients.  First, the ambient conjugacy
carries the polar divisor to an algebraic normal-crossings invariant divisor.
Second, the theorem of Camacho--Lins Neto--Sad on topological invariance and
equidesingularization of generalized-curve singularities (\cite{CamachoLinsNetoSad1984}) supplies the local
hypotheses required by Licanic's global logarithmicity criterion (\cite{Licanic2000}).

\begin{lemma}[Preservation of logarithmic crossing singularities]\label{lem:C-sing}
Let $p\in D\cap\Sing(\Gcal)$ and put $q=h(p)$.  Then
\[
 q\in\Sing(\Fcal).
\]
\end{lemma}

\begin{proof}
Since $D$ is reduced normal crossings and all residues $\lambda_j$ are
nonzero, a point belonging to only one smooth component of $D$ is regular for
$\Gcal$.  Indeed, if locally $D_j=\{x=0\}$, then
\[
 \Omega_0=\lambda_j\frac{dx}{x}+\eta,
\]
with $\eta$ holomorphic, and multiplication by $x$ gives a holomorphic
defining form whose value along $x=0$ is $\lambda_j\,dx\neq0$.
Consequently every point of $D\cap\Sing(\Gcal)$ is a crossing of two
components.  In suitable local coordinates centered at $p$,
\(
 D=\{xy=0\},
\)
and the two branches are distinct separatrix leaves through $p$.

Assume by contradiction that $q=h(p)$ is a regular point of $\Fcal$.  In a
flow box centered at $q$, exactly one punctured plaque has $q$ in its
closure.  The two distinct source separatrix leaves through $p$ are sent by
$h$ to two distinct leaves of $\Fcal$ whose closures contain $q$, which is
impossible in a flow box.  Hence $q$ is singular.
\end{proof}

\begin{proof}[Proof of Theorem~\ref{thm:B}]
For each irreducible component $D_j\subset D$ let
\[
 D_j^\circ=D_j\setminus\Sing(\Gcal),
 \qquad L_j'=h(D_j^\circ),
 \qquad D_j'=h(D_j).
\]
Since $D_j^\circ$ is a leaf of $\Gcal$, $L_j'$ is a leaf of $\Fcal$.
Since $D_j$ is compact and $D_j^\circ$ is dense in $D_j$,
\[
D_j'=\overline{L_j'}.
\]
Away from the finitely many images of the singular points of $\Gcal$ on
$D_j$, the set $L_j'$ is an embedded complex analytic curve.  By the
Remmert--Stein extension theorem (\cite{Chirka1989,RemmertStein1953}) its closure across those isolated points is
analytic.  Hence every $D_j'$ is an analytic invariant curve.  Chow's theorem (\cite{GriffithsHarris1978})
then implies that it is algebraic, and
\[
 D'=D_1'\cup\cdots\cup D_r'
\]
is a reduced algebraic invariant divisor.

At a crossing $p\in D_i\cap D_j$, the germ $(D_i\cup D_j,p)$ consists of
two smooth branches with intersection multiplicity one.  The ambient
homeomorphism gives an embedded topological equivalence with the image germ.
For reduced plane curve germs the number of branches, the topological type of
each branch and the pairwise intersection multiplicities are embedded
topological invariants; see, for instance, \cite{Wall2004}.  Therefore the
image branches are smooth and transverse.  Thus $D'$ has normal crossings.

Fix now $p\in D\cap\Sing(\Gcal)$ and put $q=h(p)$.
By Lemma~\ref{lem:C-sing}, $q\in\Sing(\Fcal)$.  In local coordinates the
source germ is represented by
\(
 \lambda_i\frac{dx}{x}+\lambda_j\frac{dy}{y}=0,
\)
or equivalently by
\(
 x\,dy-\mu y\,dx=0,
 \qquad
 \mu=-\frac{\lambda_i}{\lambda_j}.
\)
The hypothesis
$\lambda_i/\lambda_j\notin\Q_{<0}$ is exactly the condition
$\mu\notin\Q_{>0}$.  Hence the source germ is a non-dicritical reduced
singularity, therefore a generalized-curve germ, and its complete separatrix
set consists of the two coordinate axes.

By the equidesingularization theorem of Camacho--Lins Neto--Sad
\cite{CamachoLinsNetoSad1984}, a germ topologically equivalent to a
generalized curve is again a generalized curve and has an isomorphic
resolution.  In particular the target germ $(\Fcal,q)$ has exactly two
convergent separatrix branches.  The two local branches of $D'$ through $q$
are already distinct invariant analytic curves, so they exhaust the complete
local separatrix set of $(\Fcal,q)$.

We can now apply Licanic's logarithmicity theorem
\cite[Theorem~A]{Licanic2000}.  On $\PP^2$ the Picard group is isomorphic to
$\mathbb Z$; every local separatrix of $\Fcal$ through a singular point of
$D'$ is a branch of $D'$ by the preceding paragraph; and every such
singularity is a generalized curve.  All hypotheses of Licanic's theorem are
therefore satisfied, and $\Fcal$ is logarithmic on the whole projective
plane.
\end{proof}

\begin{proof}[Proof of Corollary~\ref{cor:surface-rigidity}]
The proof is the same.  The image $D'=h(D)$ is a compact invariant analytic
normal-crossings divisor; since $X$ is algebraic it is algebraic.  At each
singular point of $D$, the source germ is a generalized curve with complete
separatrix set equal to the local branches of $D$.  The theorem of
Camacho--Lins Neto--Sad transfers the generalized-curve property and the
number of separatrix branches to the target germ, so the local branches of
$D'$ form its complete separatrix set.  Since
$\operatorname{Pic}(X)\simeq\mathbb Z$, Licanic's theorem applies and yields
that $\Fcal$ is logarithmic on $X$.
\end{proof}

\begin{proof}[Proof of Corollary~\ref{cor:residue-rigidity}]
If
\(
 \frac{\lambda_i}{\lambda_j}\in\R\setminus\Q
\)
whenever $D_i\cap D_j\neq\varnothing$, then in particular
$\lambda_i/\lambda_j\notin\Q_{<0}$.  Thus every crossing is non-dicritical,
and Theorem~\ref{thm:B} applies directly.
\end{proof}

\begin{remark}\label{rem:irrational-sufficient}
The irrationality assumption in Corollary~\ref{cor:residue-rigidity} is a
sufficient condition, not a necessary one.  What must be excluded at a
crossing is the dicritical case, corresponding to 
\(
 \lambda_i/\lambda_j\in\Q_{<0}.
\)
\end{remark}

The following elementary example shows that the residue condition in
Corollary~\ref{cor:residue-rigidity} is nonempty and occurs naturally already
for the simplest normal-crossings divisor in $\PP^2$, namely the coordinate
triangle.

\begin{example}
For the coordinate triangle in $\PP^2$, take
\(
 \Omega=\frac{dZ_0}{Z_0}+\sqrt2\,\frac{dZ_1}{Z_1}
 -(1+\sqrt2)\frac{dZ_2}{Z_2}.
\)
The projective residue relation is satisfied, every pairwise residue ratio is
real and irrational, and the polar divisor is a normal-crossings curve.
Corollary~\ref{cor:residue-rigidity} therefore applies to every foliation
ambiently topologically conjugate to this logarithmic model.
\end{example}

\end{document}